\documentclass[10pt]{amsart}

\usepackage[T1]{fontenc}
\usepackage[utf8]{inputenc}

\usepackage{commath}
\usepackage{mathtools}
\usepackage{graphicx}
\usepackage{enumitem}
\usepackage{pb-diagram}
\usepackage{amsfonts,amssymb,amsthm,amscd,amsmath}
\usepackage{latexsym,stmaryrd}
\usepackage{lscape}
\usepackage[all]{xy}
\usepackage{setspace}

\usepackage[margin=1.1in]{geometry}

\usepackage{bussproofs}

\usepackage[colorlinks,linkcolor=blue,citecolor=magenta]{hyperref}
\usepackage{hyperref}

\theoremstyle{plain}
\newtheorem{theorem}{Theorem}[section]
\newtheorem{lemma}[theorem]{Lemma}
\newtheorem{proposition}[theorem]{Proposition}
\newtheorem{corollary}[theorem]{Corollary}
\newtheorem{conjecture}[theorem]{Conjecture}
\newtheorem{problem}[theorem]{Problem}

\theoremstyle{definition}

\newtheorem{remark}[theorem]{Remark}

\newcommand{\spinc}{\text{Spin}^{c}}

\usepackage{tikzsymbols}
\usepackage{textcomp}

\newcommand{\lang}{\langle}
\newcommand{\rang}{\rangle}

\newcommand{\pa}{\partial}

\newcommand{\al}{\alpha}
\newcommand{\be}{\beta}
\newcommand{\e}{\epsilon}
\newcommand{\om}{\omega}
\newcommand{\Om}{\Omega}

\newcommand{\z}{\zeta}

\newcommand{\bB}{\mathbb{B}}
\newcommand{\bC}{\mathbb{C}}

\newcommand{\bN}{\mathbb{N}}

\newcommand{\bR}{\mathbb{R}}

\newcommand{\cH}{\mathcal{H}}
\newcommand{\cI}{\mathcal{I}}

\newcommand{\fK}{\mathfrak{K}}

\newcommand{\fT}{\mathfrak{T}}

\newcommand{\ra}{\rightarrow}

\newcommand{\sub}{\subseteq}

\begin{document}
\title{Perturbations of principal submodules in the Drury-Arveson space} 

\author{Mohammad Jabbari}
\address{Mohammad Jabbari, Centro de Investigacion en Matematicas, A.P. 402, Guanajuato, Gto., C.P. 36000, Mexico}
\email{mohammad.jabbari@cimat.mx}

\author{Xiang Tang}
\address{Xiang Tang, Department of Mathematics and Statistics, Washington University in St. Louis, St. Louis, Missouri, 63130, USA}
\email{xtang@wustl.edu}

\maketitle

\begin{abstract} 
We study the geometry in the perturbations of principal submodules in the Drury-Arveson space. We show that the perturbations give rise to smooth vector bundles of Hilbert spaces which are equipped with natural Hermitian connections. We compute the associated parallel transport  operators and explore properties of the monodromy. 
\end{abstract}

\section{Introduction}

Let $A:=\bC[z_1,\ldots,z_m]$ be the algebra of polynomials in $m$ variables, and $I$ be an ideal of $A$. The zero variety of $I$ is the algebraic set $V(I):=\{z\in \mathbb{C}^m:f(z)=0, \forall p\in I\}$. The quotient $Q_I:=A/I$ can be viewed as the algebra of algebraic functions on $V(I)$. In algebraic geometry, one studies the geometry of $V(I)$ by investigating the properties of the ideal $I$ and the quotient $Q_I$. 

At the beginning of this century, Arveson \cite{arveson-curvature,arveson-dirac,arveson-standard} and Douglas \cite{bd,douglas-index} introduced an intriguing connection between multivariate operator theory and algebraic geometry. Let $H^2_m$ be the Drury-Arveson space, $\bar{I}$ be the closure of $I$ in $H^2_m$, and $I^\perp$ be the orthogonal complement of $I$ in $H^2_m$. The quotient Hilbert space $\mathcal{Q}_I:=H^2_m/ \bar{I}$, which is isomorphic to $I^\perp$, can be viewed as the Drury-Arveson space $H^2(\Omega_I)$ associated to the set $\Omega_I:=\mathbb{B}^m\cap V(I)$. Let $P_{I^{\perp}}$ be the orthogonal projection in $H^2_m$ onto $I^{\perp}$. For each $ p\in A$, define 
\[
T_p:=P_{I^\perp} M_p |_{I^\perp}
\]
to be the compression of the multiplication operator $M_p$ on $H^2_m$.  $T_p$ can be viewed as a Toeplitz operator on $H^2(\Omega_I)$. Arveson and Douglas in a series of articles proposed to investigate the operators $\{T_p: p\in A\}$ on $\mathcal{Q}_I$ to understand the geometry of $\Omega_I$. More precisely, they focused on the following index problems. 

\begin{conjecture}[Arveson \cite{arveson-dirac,arveson-conjecture}]\label{conjecture-a}
All commutators $[T_{z_j},T_{z_k}^{*}]$, $j,k=1,\ldots,m$ are compact.
\end{conjecture}

Let $\mathfrak{K}$ be the ideal of compact operators on $I^\perp$ and $\mathfrak{T}_I$ be the unital $C^*$-algebra generated by $\{T_p: p\in A\}\cup \mathfrak{K}$. Suppose that the conjecture above holds true. It follows that the quotient $\mathfrak{T}_I/\mathfrak{K}$ is a commutative $C^*$-algebra, and can be identified as $C(\sigma^e_I)$, where $\sigma^e_I$ is the essential Taylor spectrum of $(T_{z_1}, ..., T_{z_m})$. Furthermore, when $I$ is homogeneous, $\sigma^e_I$ can be identified as $X_I:=V(I)\cap\pa\bB^m$, c.f. 
\cite[Corollary 3.10]{curto}, \cite{grs2}, \cite[Theorem 5.1]{gw}. In summary, the Arveson conjecture \ref{conjecture-a} gives the following exact sequence of $C^*$-algebras, 
\[
0\ra \fK\hookrightarrow \fT_I\ra C(\sigma^e_I)\ra 0.
\]
By the Brown-Douglas-Fillmore theory \cite{bdf1,bdf2}, such an extension of $C(\sigma^e_I)$ defines an odd $K$-homology class $\tau_I$ in the K-homology group $K_1(\sigma^e_I)$. Douglas asked for an explicit computation of this element in the geometric realization of $K$-homology; more specifically, he conjectured that:

\begin{conjecture}[Douglas \cite{douglas-index}]\label{conjecture-d}
	Let $I$ be the vanishing ideal of an algebraic set $V\sub\bC^{m}$ which intersects $\pa\bB^m$ transversally. 
	Then, Conjecture \ref{conjecture-a} holds true, and its induced extension class $\tau_I$ is identified with the fundamental class of $X_I$, namely the extension class induced by the $\spinc$ Dirac operator associated to the natural Cauchy-Riemann structure of $X_I$.
\end{conjecture}

There have been many studies on Conjectures \ref{conjecture-a} and \ref{conjecture-d} in the last decades. A survey of results about these conjectures is given in \cite[Chapter 41]{alpay} and \cite{gw-survey}. 

In this article, we propose to study the ideal $I$ and $\mathcal{Q}_I$ by perturbation. Our study is inspired by the beautiful work of Brieskorn and Milnor \cite{brieskorn,hirzebruch, milnor,looijenga} which connects singularity theory, algebraic geometry, and differential topology. Let us take the following example of a polynomial,
\[
	f_j:=z_1^2+z_2^2+z_3^2+z_4^3+z_5^{6j-1}\in\bC[z_1,\ldots,z_5].
\]
Consider the zero variety $V\big(f_j-\epsilon \exp(i t)\big)$. Fix a sufficiently small $\epsilon>0$. It is not a difficult fact to check in differential geometry that the zero variety $V\big(f_j-\epsilon \exp(i t)\big)$ is smooth inside the unit ball $\mathbb{B}^5$, and it intersects transversely with the sphere $\mathbb{S}^{9}$ as the boundary of $\mathbb{B}^5$ in contrast to the property that the zero variety $V(f_j)$ has an isolated singularity at the origin. The intersection $X_j:=\mathbb{S}^9\cap V\big(f_j-\epsilon \exp(it)\big)$ is a smooth manifold homeomorphic to $\mathbb{S}^7$ but not necessarily diffeomorphic to the smooth structure on the standard sphere. Actually, when $j$ runs through 1 to 28, $X_j$ gives all distinct oriented smooth structures on the topological 7-sphere. 

We hope to learn from the success above in differential topology to study the Hilbert modules. As an experiment of this idea, we study a  special class of ideals in $H^2_2$ in this article. More precisely, take $I_k:=\langle z_1^k\rangle$ to be the principal ideal in $\mathbb{C}[z_1, z_2]$ generated by the monomial $f_k:=z_1^k$. Consider the perturbation $I_k^{\epsilon, t}$ of the ideal $I_k$ by varying the generator $f_k^{\epsilon, t}:=z_1^k-\epsilon \exp(it)$, for a sufficiently small $\epsilon>0$. The collection $\left\{\big(I^{\epsilon, t}_k \big)^\perp\right\}_{t\in [0, 2\pi)}$ forms a family of closed Hilbert subspaces of $H^2_2$ parameterized by $t\in \mathbb{R}/2\pi \mathbb{Z}$. In Section \ref{sec:power}, we will present the following results about the variation $\left\{I_k^{\epsilon, t}\right\}$. 

\begin{enumerate}
\item  (Theorem \ref{frame}) The family $\left\{\mathcal{Q}_{I^{\epsilon, t}_k}\right\}_{t\in [0, 2\pi)}$  forms a smooth vector bundle $\mathcal{Q}_{k, \epsilon}$ of Hilbert spaces over the circle $\mathbb{S}^1:=\mathbb{R}/2\pi\mathbb{Z}$.
\item (Theorem \ref{thm:holonomy-u} and Theorem \ref{conjugation}) The vector bundle $\mathcal{Q}_{k, \epsilon}$ naturally embeds in the trivial bundle $\mathbb{S}^1\times H^2_2\to \mathbb{S}^1$, and is equipped with a natural connection. We compute the parallel transport with respect to this connection, and the associated monodromy operator $U$. 
\item (Theorem \ref{diffloss}) Let $P_{t}$ be the orthogonal projection of $H^2_2$ onto $\big(I^{\epsilon, t}_k\big)^\perp $. Though $P_t$ is not differentiable with respect to $t$ as a family of operators on the Drury-Arveson space, it is a differentiable family as maps between the Besov-Sobolev spaces.  
\end{enumerate}

Our results above are only about a special class of ideals in $H_2^2$. The success on these examples encourage us to seek a general theory of perturbation of principal submodules. We briefly discuss in Section \ref{sec:outlook} a few questions for the general cases and will study them more systematically in future publications.  \\

\noindent{\bf Acknowledgements:} We would like to thank Ronald Douglas, Kai Wang and Guoliang Yu for inspiring discussions. Both authors' research are partially supported by National Science Foundation. The works in this article were reported in the first author's Ph.D. thesis, \cite{jabbari-thesis}. 

\section{Perturbation of $\langle z_1^k\rangle $}
\label{sec:power}

Set 
\[ 
I(t):=\lang z_1^k-\e e^{it}\rang,\quad \cI^\perp:=\biguplus\left\{ \lang z_1^k-\e e^{it}\rang^{\perp}\sub  H_2^2:t\in\bR\right\}\sub\bR\times H_2^2.
\]
Let
\[
p:\mathcal{I}^{\perp}\ra\bR,
\quad
p\left(I(t)^\perp\right)=\{t\},
\]
and 
\[
P:\bR\ra B\left( H^2_2 \right),\quad
P:=(P_t)
\]
be respectively the assembly of Hilbert spaces $I(t)^\perp$ and orthogonal projections $P_t: H^2_2\to  \lang z_1^k-\e e^{it}\rang^{\perp} $ into a smooth\footnote{The smoothness of the vector bundle will be established in Theorem \ref{frame}.} Hilbert bundle and a rough\footnote{Namely, we are momentarily putting aside continuity or smoothness considerations.}  map between Banach spaces.
Topologize $\cI^\perp\sub\bR\times H_2^2$ with the subspace topology.

\subsection{Smooth vector bundle $\mathcal{I}^\perp$}\label{z1k}
We first find an explicit smooth orthonormal frame for our Hilbert bundle $\cI^{\perp}$.

\begin{lemma}\label{lm}
	Let $E$ be a complex number with $|E|<1$.
	Set 
	\[
	F:=E^{\frac{1}{k}},
	\]
	\[
	\z_{j}:=e^{i\frac{2\pi}{k}j},\quad j=0,\ldots,k-1,
	\]
	\[
	a_{j}:=1-\z_{j}F,\quad j=0,\ldots,k-1.
	\]
	
	(a)
	We have
	{\small
	\[
	\sum_{q\in\bN} \binom{n+r+k q}{n}E^{ q}
	=k^{-1}F^{-r}\sum_{j=0}^{k-1}\z_{j}^{-r}a_j^{-n-1},
	\]
	\[
	\sum_{q\in\bN} \binom{n+r+k q}{n}E^{ q}q
	=k^{-2}F^{-r}\left(-r\sum_{j=0}^{k-1}\z_{j}^{-r}a_j^{-n-1}+F(n+1)\sum_{j=0}^{k-1}\z_{j}^{-r+1}a_j^{-n-2}\right),
	\]
	\begin{multline*}
	\sum_{q\in\bN} \binom{n+r+k q}{n}E^{ q}q^2
	=k^{-3}F^{-r}\times\\
	\left(r^2\sum_{j=0}^{k-1}\z_{j}^{-r}a_j^{-n-1}+F(1-2r)(n+1)\sum_{j=0}^{k-1}\z_{j}^{-r+1}a_j^{-n-2}+F^2(n+1)(n+2)\sum_{j=0}^{k-1}\z_{j}^{-r+2}a_j^{-n-3}\right).
	\end{multline*}}
	
	(b)
Given nonnegative integer $l$, we have the asymptotic formula\footnote{Two sequences $a_n$ and $b_n$ are said to be asymptotically equivalent, denoted by $a_n\approx b_n$, if there is a finite nonzero number $C$ such that $
	\lim_{n\to \infty}a_n/b_n=C$. 
}
	\[
	\sum_{q\in\bN} \binom{n+r+k q}{n}E^{q}q^l\approx n^{l}(1-F)^{-n},
	\]
	as $n\ra\infty$. \end{lemma}

\begin{proof}
	(a)
	Note that the sequence of numbers
	\[
	\psi_{ q}:=k^{-1}\sum_{j=0}^{k-1}\z_j^{ q-r},\quad q\in\bN
	\]
	equals $1$ when $ q$ has remainder $r$ modulo $k$, and zero otherwise.
	Therefore, we have the following equation, 
	\begin{equation}
	\label{eq:sum}
	\begin{aligned}
	\sum_{q\in\bN} \binom{n+r+k q}{n}E^{r+k q}
	=
	\sum_{q\in\bN} \binom{n+ q}{n}\psi_{ q}E^{ q}
	&=
	k^{-1}\sum_{q\in\bN}\sum_{j=0}^{k-1} \binom{n+ q}{n}\z_{j}^{ q-r}E^{ q}\\
	&=
	k^{-1}\sum_{j=0}^{k-1}\z_{j}^{-r}\left(1-\z_j E\right)^{-n-1},
	\end{aligned}
	\end{equation}
	where in the last line we have used the negative binomial formula
	\begin{equation*}
	\sum_{q\in\bN} \binom{n+ q}{n}G^{ q}=(1-G)^{-n-1}\label{binom}.
	\end{equation*}
	
	Equation (\ref{eq:sum}) gives the first formula. The other two are followed by differentiation with respect to $E$. 
	
	(b) A straightforward induction on $l$ shows that the left hand side is of the form 
\begin{equation}
\sum_{m=0}^{l} A_mF^{-r}(n+1)\cdots(n+m)\left(\sum_{j=0}^{k-1}\z_{j}^{-r+m}a_j^{-n-m-1}\right),\label{whole}
\end{equation}
where each $A_m$ depends only on $k, l, r$ but not $E$, $F$ or $n$.
When $n\ra\infty$, the dominant summand in each $\sum_{j=0}^{k-1}\z_j^{-r+m}a_j^{-n-m-1}$ is the one with the smallest $|a_j|$, namely the one with $j=0$.
Therefore, the dominant term in (\ref{whole}) is the one with $m=l$ and $j=0$. 
\end{proof}

Having this lemma at hand, we construct an orthogonal frame for $\cI^\perp$. 

\begin{theorem}\label{U} 
	Set 
	\[
	F:=\e^{\frac{2}{k}},
	\]
	\[
	\z_{j}:=e^{i\frac{2\pi}{k}j},\quad j=0,\ldots,k-1,
	\]
	\[
	a_{j}:=1-\z_{j}F,\quad j=0,\ldots,k-1,
	\]
	\[
	J:=\left\{(r,n)\in\bN^2:0\leq r\leq k-1\right\}.
	\]

	(a)
	A smooth orthogonal frame for the Hilbert bundle $\cI^{\perp}$ is given by 	
	\begin{equation}
	\al:=\left\{\al_{r,n}(t):=
	\sum_{q\in\bN}\om_{r+kq,n}^{-1}\e^{ q}e^{-i q t}z_1^{r+k q}z_2^n \ : \ (r,n)\in J, \  t\in\bR\right\},\label{alpharn}
	\end{equation}
	where
	\[
	\om_{m,n}=\|z_1^mz_2^n\|^{2}_{H^{2}_2}=\binom{m+n}{m}^{-1}.
	\]
	
	(b)
	A smooth orthonormal frame for the Hilbert bundle $\cI^{\perp}$ is given by 	
	\begin{equation}
	\be:=\left\{\be_{r,n}(t):=\frac{\al_{r,n}(t)}{\|\al_{r,n}(t)\|} \ : \ (r,n)\in J, \  t\in\bR\right\},\label{betarn}
	\end{equation}
	where  
	\begin{equation}
	\|\al_{r,n}(t)\|^2
	=\sum_{q\in\bN}\om_{r+kq,n}^{-1}\e^{2q}
	=F^{-r}k^{-1}\sum_{j=0}^{k-1}\z_{j}^{-r}a_{j}^{-n-1}.\label{norm}
	\end{equation}
	
\end{theorem}

\begin{proof}
	(a)
	We first check the smoothness of the frame.
	For comparison purposes, observe that any one-variable power series of the form 
	\begin{align}
	\sum_{q\in\bN} R(q) \z^{ q},\quad
	R\in\bC[\z]\ \text{a polynomial in single variable}\ \z\label{comparison}
	\end{align}
	has the radius of convergence equal to one, hence absolutely and uniformly convergent on any compact subset of the open unit disk of the complex $\z$-plane.  
	The formal power series of term-by-term time derivative of each $\al_{r,n}$ of order $l\in\bN$, as well as its $H_2^2$-norm are given by:
	\begin{equation*}
	\frac{d^l\al_{r,n}}{dt^l}:=
	\sum_{q\in\bN}\binom{kq+r+n}{n}(-iq)^le^{-i q t}\e^{q}z_1^{r+k q}z_2^n,
	\end{equation*}
	\begin{equation*}
	\left\|\frac{d^l\al_{r,n}}{dt^l}\right\|^2_{H_2^2}:=
	\sum_{q\in\bN}\binom{kq+r+n}{n}q^{2l}\e^{2q}.
	\end{equation*}
	Comparison with (\ref{comparison}) shows that for any $\e<1$ and $t\in\bR$, each $d^l\al_{r,n}/dt^l(t)$ is an analytic function on $\bB^2$ with finite $H_2^2$-norm, hence it lives in $H^2_2$.
	That $\al_{r,n}$ lives in $I(t)^{\perp}$ is immediate from our derivation of $\al_{r,n}$ in the next paragraph, but here is a direct verification.
	For each $(M,N)\in\bN^2$, $\al_{r,n}(t)$ and $z_1^Mz_2^N\left(z_1^k-\e e^{it}\right)$ have no monomial in common (hence orthogonal) except when $N=n$ and $r$ equals the remainder of $M$ in division by $k$. 
	For this exceptional case, assuming $M=kQ+r$, $Q\in\bN$, we have
	\[
	\left\lang\al_{r,n},z_1^Mz_2^N\left(z_1^k-\e e^{it}\right)\right\rang=\e^{Q+1}e^{-i(Q+1)t}-\e^{Q}e^{-iQ t}\e e^{-it}=0.
	\]
	
	By Taylor's theorem, we have
	\[
	\left\|\al_{r,n}(t+h)-\al_{r,n}(t)-h\frac{d\al_{r,n}}{dt}(t)\right\|^2
	=
	\sum_{q\in\bN}\binom{kq+r+n}{n}\e^{2q}\left|e^{-iq (t+h)}-e^{-iqt}+hiqe^{-iqt}\right|^2
	\]	
	\[
	\times\left(\sum_{q\in\bN}\binom{kq+r+n}{n}\e^{2q}\left(\frac{h^2}{2!}q^2\right)^2\right),
	\]
	which shows that $\al_{r,n}:\bR\ra H_2^2$ is first-order differentiable. 
	The same line of arguments proves the smoothness as well.
	
	Next, we show that the sections of $\cI^{\perp}$ are linear combinations of $\al_{r,n}$.
	A section of $\cI^{\perp}$ has the form
	\begin{equation}
	\xi(t)=\sum_{m,n\geq 0} x_{m,n}(t)z_1^mz_2^n,\label{xi}
	\end{equation}
	and satisfies the following orthogonality equations:
	\[
	0=\left\lang\xi(t),z_1^mz_2^n\left(z_1^k-\e e^{it}\right)\right\rang=x_{m+k,n}\om_{m+k,n}-x_{m,n}\om_{m,n}\e e^{-it},\quad\forall m,n\geq 0,
	\]
	or equivalently
	\begin{equation}
	x_{m+k,n}\om_{m+k,n}=x_{m,n}\om_{m,n}\e e^{-it},\quad\forall m,n\geq 0.\label{recursive}
	\end{equation}
	Assuming
	\[
	X_{m,n}:=x_{m,n}\om_{m,n},
	\]
	this latter recursive equation becomes 
	\[
	X_{m+k,n}=X_{m,n}\e e^{-it},
	\]
	hence
	\begin{equation}
	X_{r+kq,n}=X_{r,n}\e^{q} e^{-iq t},\quad
	r=0,1,\ldots,k-1,\quad q,n=0,1,2,\ldots.\label{closed}
	\end{equation}
	This shows that
	\[
	\left\{X_{r,n}:(r,n)\in J\right\}
	\]
	are basic Taylor coefficients of $\xi$ in the sense that they linearly determine all the other coefficients, and there are no nontrivial linear equations among them. 
	Note that $\al_{r,n}$ is the section with $X_{r,n}=1$, and all other basic coefficients vanish.
	Working backwards, one can formally (neglecting convergence issues) write any element $\xi(t)$ of $I(t)^\perp$ as a linear combination of elements (\ref{alpharn}).
	That the closed linear span of (\ref{alpharn}) equals $I(t)^\perp$ can be easily proved by checking that  $\xi(t)=0$ is implied by assuming $\lang\xi(t),\al_{r,n}(t)\rang=0$ for all $(r,n)\in J$.
	Any two $\al_{r,n}$ and $\al_{r',n'}$ with $(r,n)\neq (r',n')$ are orthogonal because they have no monomials in common, and we know that monomials constitute an orthogonal basis for $H_2^2$.  
	
	(b)
	Lemma \ref{lm} gives Equation (\ref{norm}).
	Since $\left\|\al_{r,n}(t)\right\|$ does not depend on $t$, the rest follows immediately from Part (a).
\end{proof}

\begin{theorem}\label{frame}
The map $p:\mathcal{I}^{\perp}\ra\bR $ defines a smooth bundle of Hilbert spaces over $\mathbb{R}$ and therefore on $\mathbb{S}^1:=\mathbb{R}/\mathbb{Z}$. 
\end{theorem}
\begin{proof}
	Assume $\cI^\perp\sub\bR\times H_2^2$ with the subspace topology as a rough Hilbert bundle over $\bR$.
	Recall that $J:=\left\{(r,n)\in\bN^2:0\leq r\leq k-1\right\}$ is the index set of the  orthonormal frame $\be$ in Theorem \ref{U}.  
	Since each $\be(t)$, $t\in\bR$ is an orthonormal basis for the fiber $I(t)^\perp$, the mapping
	\[
	\Phi:\bR\times l^2(J)\ra\cI^\perp,\quad
	\left(t,(a_{r,n})_{(r,n)\in J}\right)\mapsto\left(t,\sum_{(r,n)\in J} a_{r,n}\be_{r,n}(t)\right)
	\]
	trivializes $\cI^\perp$ as a topological vector bundle, namely, $\Phi$ is a homeomorphism and the triangle 
	\[
	\xymatrix{
		\bR\times l^2(J)\ar[r]^{\ \ \ \Phi}\ar[d]_{\mathrm{pr}_{\bR}} & \cI^\perp \ar[dl]^p\\\bR}
	\]
	commutes.
	Since this trivialization is given by a single chart, it also gives $\cI^\perp$ the structure of a smooth vector bundle.
\end{proof}

\subsection{Parallel transport $U_t$}

We observe that $\cI^\perp$ is a closed subbundle of the trivial vector bundle $\mathbb{S}^1\times H^2_2\to \mathbb{S}^1$. Recall that $P:\mathbb{R}\to B(H^2_2)$ is the orthogonal projection from $\mathbb{S}^1\times H^2_2$ to $\mathcal{I}^\perp$. 

The imitation of the standard construction of the Hermitian connection for subbundles of Hilbert bundles \cite[Example 1.5.14]{klingenberg}, \cite[Volume II, Page 540]{taylor-PDE} defines a covariant derivative:
\begin{equation}\label{connection-toy}
D\xi(t)=P_t\left(\frac{d\xi}{dt}\right),\quad\xi\in C^{\infty}(\bR;\cI^\perp),
\end{equation}
where $C^{\infty}(\bR; \cI^\perp)$ is the space of smooth functions on $\bR$ with value in $\cI^\perp$. 

The operator $D$ is called a covariant derivative because it satisfies the Leibniz rule
\[
D(g\xi)(t)=g'(t)\xi(t)+g(t)D(\xi)(t),\quad
\forall g\in C^{\infty}(\bR;\bC),\quad
\forall \xi\in  C^{\infty}(\bR;\cI^{\perp}).
\]
A $D$-flat section of the bundle $\cI^\perp$ is a section $\xi\in C^{\infty}(\bR;\cI^{\perp})$ which satisfies the evolution equation
\begin{equation}
	D\xi(t)=0,\quad \forall t\in\bR.\label{ODE-toy}
\end{equation}

We prove below that Equation (\ref{ODE-toy}) has a unique solution operator for all $t\in \mathbb{R} $ for any initial value $\xi (0)\in I(0)^\perp $. This leads to the solution map
\[
U_t:=I(0)^{\perp}\ra I(t)^{\perp},\quad t\in\bR, 
\]
which sends the initial value $\xi(0)$ of a flat section $\xi$ to its time-$t$ value $\xi(t)$.

\begin{theorem}\label{thm:holonomy-u} Assume the notations in Theorem \ref{frame}. 
	\begin{enumerate}
	\item[(a)]	The solution operator $U_t:I(0)^{\perp}\ra I(t)^{\perp}$ acts diagonally by
	\begin{equation}
	U_t\left(\be_{r,n}(0)\right)
	=e^{i f_{r,n}t}\be_{r,n}(t),\label{U1}
	\end{equation}
	where frequencies $f_{r,n}$ are given by
	\begin{equation}
	f_{r,n}
	=
	\frac{\sum_{q\in\bN}\om_{r+kq,n}^{-1}q \e^{2q}}{\sum_{q\in\bN}\om_{r+kq,n}^{-1}\e^{2q}}
	=
	\frac{-r\sum\limits_{j=0}^{k-1}\z_{j}^{-r}a_{j}^{-n-1}+ F(n+1)\sum\limits_{j=0}^{k-1}\z_{j}^{-r+1}a_{j}^{-n-2}}{k\sum\limits_{j=0}^{k-1}\z_{j}^{-r}a_{j}^{-n-1}}.\label{frn}
	\end{equation}
	\item[(b)] The operator $U_t: I(0)^\perp\to I(t)^\perp$ is unitary, $\forall t\in \mathbb{R}$.
	\item[(c)] When $n\ra\infty$, $f_{r,n}$ varies asymptotically like	\begin{equation} \frac{F}{k(1-F)}n+\frac{F}{k(1-F)}-\frac{r}{k}.\label{infty}
	\end{equation}
	\item[(d)] A smooth orthonormal parallel frame for the Hilbert bundle $\cI^{\perp}$ is given by 
	\begin{equation}
	\gamma:=\left\{\gamma_{r,n}(t):=e^{i f_{r,n}t}\be_{r,n}(t) \ : \ (r,n)\in J, \  t\in\bR\right\}.\label{gamma}
	\end{equation}
	\end{enumerate}
\end{theorem}

\begin{proof}
	(a)
	A flat section of $\cI^{\perp}$ has the form $\eta(t)=\sum y_{m,n}z_1^mz_2^n$ such that $\dot{\eta}=\sum \dot{y}_{m,n}z_1^mz_2^n$ lives in $(I(t)^{\perp})^\perp$ for each $t$.
	In other words, the inner product $\lang\dot{\eta},\xi\rang_{H^2_2}$ is zero for every section $\xi$ as in (\ref{xi}).
	Equivalently, in terms of Taylor coefficients, we have
	\[
	\sum_{m,n} \dot{y}_{m,n}\overline{X}_{m,n}=0
	\]
	for all $X_{m,n}$ satisfying (\ref{closed}).
	Rewriting this in terms of basic Taylor coefficients, we get
	\[
	\sum_{\substack{0\leq r<k\\ q,n\geq 0}} \dot{y}_{r+k q,n}\overline{X}_{r,n}\e^{ q} e^{i q t}=0.
	\]
	Since this is true for any choice of basic coefficients $X_{r,n}$, $(r,n)\in J$, we should have
	\begin{equation}
	\sum_{q\in\bN} \dot{y}_{r+k q,n}\e^{ q} e^{i q t}=0,\quad (r,n)\in J.\label{orthoality}
	\end{equation}
	Since $\eta$ is a section, its Taylor coefficients satisfy
	\begin{equation*}
	y_{r+k q,n}=y_{r,n}\frac{\om_{r,n}}{\om_{r+k q,n}}\e^{ q} e^{-i q t},\quad (r,n)\in J,\quad q\in\bN.
	\end{equation*}
	(Recall (\ref{recursive}).)
	Plugging this into (\ref{orthoality}) yields 
	\begin{equation*}
	\sum_{q\in\bN} \left(\dot{y}_{r,n}-i q y_{r,n}\right)\frac{\om_{r,n}}{\om_{r+k q,n}}\e^{2q}=0,\quad (r,n)\in J.
	\end{equation*}
	Therefore, we have the explicit evolution laws
	\begin{equation}
	\dot{y}_{r,n}=y_{r,n}if_{r,n},\quad (r,n)\in J,\label{el}
	\end{equation}
	where 
	\begin{equation}
	f_{r,n}
	=\frac{\sum_{q\in\bN} \binom{n+r+k q}{n}q \e^{2q}}{\sum_{q\in\bN} \binom{n+r+k q}{n}\e^{2q}}.\label{freq}
	\end{equation}
	
	Evolution equations (\ref{el}) are solved as
	\[
	y_{r,n}(t)=y_{r,n}(0)e^{if_{r,n}t},\quad (r,n)\in J,
	\]
	hence (\ref{U1}).
	Lemma \ref{lm} computes $f_{r,n}$.
	
	(b) The linearity of $U_t$ is immediate from the properties that $D$ is linear and the uniqueness of the solution of Equation (\ref{ODE-toy}). Let us first prove the uniqueness of the solution. Suppose $\xi\in \cI^\perp$ satisfies $D\xi=0$ and $\xi(0)=0$. Then
	\[
	0=\langle D\xi, \xi\rangle=\left\langle \frac{d}{dt}\xi, \xi \right\rangle=\frac{1}{2}\frac{d}{dt}\langle\xi, \xi\rangle,\quad \xi(0)=0.
	\]
	This implies that $\langle \xi, \xi \rangle  \equiv 0$, and the solution of Equation (\ref{ODE-toy}) is unique. 
	
	Assume that $\xi$ is a flat section. Since $\xi$ and $d\xi/dt$ are orthogonal, we have
	\[
	0
	=2\left\lang \xi(t),\frac{d\xi}{dt}\right\rang
	=\frac{d}{dt}\left\|\xi(t)\right\|^2
	=\frac{d}{dt}\left\|U_t\xi(0)\right\|^2,
	\]
	hence
	\[
	\left\|U_t\xi(0)\right\|=\left\|U_0\xi(0)\right\|=\left\|\xi(0)\right\|.
	\]
	This shows that $U_t$ is an isometry.
	It remains to show that $U_t$ is surjective.
	Given $\tau\in\bR$, the inverse of $U_{\tau}:\xi(0)\mapsto\xi(\tau)$ is given by the parallel translation $\eta(0)\mapsto\eta(\tau)$ along the flat section $\eta(t):=\xi(\tau-t)$, where we have used the uniqueness of the solutions of Equation (\ref{ODE-toy}).
	
	(c)
	When $n\ra\infty$, the dominant summands in the numerator and denominator of $f_{r,n}$ in (\ref{frn}) are those with the smallest $|a_j|$, namely those with $j=0$.
	Therefore, $f_{r,n}$ varies asymptotically like
	\[
 \frac{-ra_0^{-n-1}+F(n+1)a_0^{-n-2}}{ka_0^{-n-1}}=-\frac{r}{k}+\frac{F}{k(1-F)}(n+1).
	\]
	
	(d)
	By (\ref{U1}) we have
	\[
	U_t(\gamma_{r,n}(0))=\gamma_{r,n}(t),
	\]
	hence $\gamma$ is a parallel frame. 
	The smoothness of $U_t$ is the result of the smoothness of $\beta$ and the asymptotic formula (\ref{infty}) for $f_{r,n}$.
\end{proof}

Consider the Toeplitz algebra $\fT_{I(0)}$ associated to the ideal $I(0)=\lang z_1^k-\e \rang\sub\bC[z_1,z_2]$.
It is the C*-algebra generated by $\left\{1, T_{z_1},T_{z_2}\right\}\cup\fK(I(0)^{\perp})$, where $T_{z_{j}}$, $j=1,2$ is the multiplication by the coordinate function $z_j$ compressed to $I(0)^{\perp}$.
For brevity, we set $T_{j}:=T_{z_{j}}$, $j=1,2$.

\begin{proposition}\label{weightedshift}
Assume the notations of Theorem \ref{U}, and set
\[
\be_{k,n}:=e^{-it}\be_{0,n},\quad
\be_{-1,n}:=e^{it}\be_{k-1,n},\quad
\be_{r,-1}:=0
\]
for every $n\geq 0$ and $0\leq r<k$.
Then, $T_1$, $T_2$ and their adjoints are weighted shifts given by
\[
T_1 \be_{r,n}
=F^{\frac{1}{2}}\frac{\left(\sum_{j=0}^{k-1}\z_{j}^{-r}a_{j}^{-n-1}\right)^{\frac{1}{2}}}{\left(\sum_{j=0}^{k-1}\z_{j}^{-r-1}a_{j}^{-n-1}\right)^{\frac{1}{2}}}
\be_{r+1,n},
\]
\[
T_1^{\ast} \be_{r,n}
=F^{\frac{1}{2}}\frac{\left(\sum_{j=0}^{k-1}\z_{j}^{-r+1}a_{j}^{-n-1}\right)^{\frac{1}{2}}}{\left(\sum_{j=0}^{k-1}\z_{j}^{-r}a_{j}^{-n-1}\right)^{\frac{1}{2}}}\be_{r-1,n},
\]
\[
T_2 \be_{r,n}
=\frac{\left(\sum_{j=0}^{k-1}\z_{j}^{-r}a_{j}^{-n-1}\right)^{\frac{1}{2}}}{\left(\sum_{j=0}^{k-1}\z_{j}^{-r}a_{j}^{-n-2}\right)^{\frac{1}{2}}}
\be_{r,n+1},
\]
\[
T_2^{\ast}\be_{r,n}
=\frac{\left(\sum_{j=0}^{k-1}\z_{j}^{-r}a_{j}^{-n}\right)^{\frac{1}{2}}}{\left(\sum_{j=0}^{k-1}\z_{j}^{-r}a_{j}^{-n-1}\right)^{\frac{1}{2}}}\be_{r,n-1}
\]
for every $0\leq r<k$ and $n\geq 0$.
\end{proposition}

\begin{proof}
We do the computations for $T_1\be_{r,n}$ when $0\leq r<k-1$. 
Each $\be_{r,n}$ is a sum of monomials $z_{1}^{r+kq}z_2^{n}$, $q\geq 0$. 
	Since distinct monomials are orthogonal to each other in $H_2^2$, $z_1\be_{r,n}$ is orthogonal to all elements $\be_{r',n'}$ of our orthonormal basis except for $\be_{r+1,n}$. Therefore, $T_1\be_{r,n}$ is just the orthogonal projection of $z_1 \be_{r,n}$ onto $\be_{r+1,n}$, namely 
	\[
	T_1 \be_{r,n}
	=\lang z_1 \be_{r,n}, \be_{r+1,n}\rang \be_{r+1,n}.
	\]

According to Theorem \ref{U} and Lemma \ref{lm}, the weight $\lang z_1 \be_{r,n},\be_{r+1,n}\rang$ equals 
\[
\frac{\sum_{q\in\bN}\om_{r+k q,n}^{-1}\e^{2q}}{F^{-r-\frac{1}{2}}k^{-1}\left(\sum_{j=0}^{k-1}\z_{j}^{-r}a_{j}^{-n-1}\right)^{\frac{1}{2}} \left(\sum_{j=0}^{k-1}\z_{j}^{-r-1}a_{j}^{-n-1}\right)^{\frac{1}{2}}}
=
F^{\frac{1}{2}}\frac{\left(\sum_{j=0}^{k-1}\z_{j}^{-r}a_{j}^{-n-1}\right)^{\frac{1}{2}}}{\left(\sum_{j=0}^{k-1}\z_{j}^{-r-1}a_{j}^{-n-1}\right)^{\frac{1}{2}}}.
\]
The rest is similar.
\end{proof}

It is immediate from Proposition \ref{weightedshift} that:

\begin{proposition}\label{prop10}
Assume the notations of Theorem \ref{U}, Theorem \ref{thm:holonomy-u} and Proposition \ref{prop10}.
Set
\[
f_{k,n}:=f_{0,n},\quad
f_{-1,n}:=f_{k-1,n},\quad U:=U_{2\pi}.
\]
for every $n\geq 0$.
Then, we have
\[
U^{\ast}T_1U\be_{r,n}
=e^{i2\pi\left(f_{r,n}-f_{r+1,n}\right)}T_1\be_{r,n},
\]
\[
U^{\ast}T_1^*U\be_{r,n}
=e^{i2\pi\left(f_{r,n}-f_{r-1,n}\right)}T_1^*\be_{r,n},
\]
\[
U^{\ast}T_2U \be_{r,n}
=e^{i2\pi\left(f_{r,n}-f_{r,n+1}\right)}T_2\be_{r,n},
\]
\[
U^{\ast}T^{\ast}_2U\be_{r,n}
=e^{i2\pi\left(f_{r,n}-f_{r,n-1}\right)}T_2^*\be_{r,n}
\]
for every $0\leq r<k$ and $n\geq 0$.

The operator $U$ is called the monodromy operator associated with the connection $D$ introduced in Equation (\ref{connection-toy}). 
\end{proposition}

We study  the asymptotic behavior of the factors appearing in Proposition \ref{prop10} when $n$ grows large. 
Recalling the asymptotic formula (\ref{infty}) for $f_{r,n}$, one expects:

\begin{lemma}\label{lemma5}
	Assume the notations of Theorem \ref{U} and Theorem \ref{thm:holonomy-u}.
	Then
\[ 
	f_{r,n}-f_{r-1,n}\ra -\frac{1}{k},\quad\quad\quad
	f_{r,n}-f_{r,n-1}\ra\frac{F}{k(1-F)}
	\]
for every $0\leq r\leq k$ as $n\ra\infty$.
\end{lemma}

\begin{proof}
	By (\ref{frn}), $k\left(f_{r,n}-f_{r,n-1}\right)$ equals
\[
		\frac{Fn\sum\limits_{j,l=0}^{k-1}\z_{j}^{-r+1}a_j^{-n-2}\z_l^{-r}a_l^{-n}-Fn\sum\limits_{j,l=0}^{k-1}\z_{j}^{-r}a_j^{-n-1}\z_l^{-r+1}a_l^{-n-1}+F\sum\limits_{j,l=0}^{k-1}\z_{j}^{-r+1}a_j^{-n-2}\z_l^{-r}a_l^{-n}}{\sum\limits_{j,l=0}^{k-1}\z_{j}^{-r}a_j^{-n-1}\z_{l}^{-r}a_l^{-n}}.
\]
	
	We need to find the dominant terms in the numerator and denominator of the latter fraction when $n$ grows large. 
	The dominant summand in the denominator is the one with the smallest $|a_j||a_l|$, which is the one with $j=l=0$, namely
	\[
	\z_{0}^{-r+1}a_0^{-n-1}\z_0^{-r}a_0^{-n}=(1-F)^{-2n-1}.
	\]
	We have three summations in the numerator with dominant terms
	\[
	Fn(1-F)^{-2n-2},\quad
	n(1-F)^{-2n-2}\quad\text{and}\quad
	F(1-F)^{-2n-2},
	\]
	respectively. 
	The first two cancel each other, and all the remaining summands in the first two summations are dominated by the dominant term of the denominator $(1-F)^{-2n-1}$. 
	Therefore, the dominant term of the numerator is $F(1-F)^{-2n-2}$, so
	\[
	\lim\limits_{n\ra\infty}
	k\left(f_{r,n}-f_{r,n-1}\right)
	=
	\lim\limits_{n\ra\infty}
	\frac{F(1-F)^{-2n-2}}{(1-F)^{-2n-1}}
	=
	\frac{F}{1-F}.
	\]
	
	Using (\ref{frn}), $k\left(f_{r,n}-f_{r-1,n}\right)$ equals
\[
		\frac{-\sum\z_{j}^{-r}\z_l^{-r+1}(a_ja_l)^{-n-1}+F(n+1)\left(\sum\z_{j}^{-r+1}a_j^{-n-2}\z_l^{-r+1}a_l^{-n}-\z_{j}^{-r}\z_l^{-r+2}(a_ja_l)^{-n-1}\right)}{\sum\z_{j}^{-r}a_j^{-n-1}\z_{l}^{-r+1}a_l^{-n-1}}.
\]

	When $n$ grows large, the dominant terms in the numerator and denominator of the latter fraction are
	\[
	-(1-F)^{-2n-2}+F(n+1)\times\big(\text{exponentially smaller than }(1-F)^{-2n-2}\big)
	\]
and
\[
(1-F)^{-2n-2},
	\]
	respectively. Therefore, $k\left(f_{r,n}-f_{r,n-1}\right)\ra -1$.
\end{proof}

Proposition \ref{prop10} and Lemma \ref{lemma5} together give:

\begin{theorem}\label{conjugation}
As before, $F:=\e^{\frac{2}{k}}$.
	The unitary operator $U:=U_{2\pi}$ acts by conjugation on the Toeplitz algebra $\fT_{I(0)}$ in the sense that $U^{\ast}\fT_{I(0)}U\sub\fT_{I(0)}$.
	In more details,
	\[
	U^{\ast}T_1U-e^{i\frac{2\pi}{k}}T_1,\quad
	U^{\ast}T_1^{\ast}U-e^{-i\frac{2\pi}{k}}T^{\ast}_1,\quad
	U^{\ast}T_2U-e^{-i \frac{2\pi F}{k(1-F)}}T_2,\quad
	U^{\ast}T^{\ast}_2U-e^{i \frac{2\pi F}{k(1-F)}}T_2^{\ast}
	\]
	are all compact.
\end{theorem}

\begin{remark}
We observe that the intersection of the zero variety of $z_1^k-\epsilon$ with the unit ball $|z_1|^2+|z_2|^2<1$ is a disjoint union $\{(\epsilon^{\frac{1}{k}}\exp(\frac{2j\pi i}{k}), z_2): |z_2|^2<1-\epsilon^{\frac{2}{k}}, 0\leq j\leq k-1\}$. Let $\fT(\mathbb{D})$ be the algebra of Toeplitz operators on $\mathbb{D}$ with continuous symbols. The Toeplitz algebra $\fT_{I(0)}$ can be identified as a direct sum of $k$ copies of $\fT(\mathbb{D})$, i.e. 
\[
\fT_{I(0)}\cong \oplus_j \fT(\mathbb{D})_{j},
\]
which is indexed by the eigenvalue of $T_{z_1}$. With Theorem \ref{conjugation}, we can directly compute that the conjugation of operator $U$ is a permutation mapping the component $\fT(\mathbb{D})_j$ to $\fT(\mathbb{D})_{j+1}$. This picture is compatible with the Milnor fibration and the open book decomposition associated to the polynomial $z_1^k$ near the isolated singularity $0$, c.f. \cite{milnor}. We plan to explore this phenomenon and its applications in the future. 
\end{remark}

\subsection{Smoothness of $P$}\label{nonsmoothness}
Recall the projection assembly map $P:\bR\ra B(H_2^2)$ acting between Banach spaces.
We now prove what we mentioned earlier:

\begin{proposition}\label{prop:nonsmooth}
	$P$ does not have a bounded derivative. 
\end{proposition}
\begin{proof}
According to Theorem \ref{thm:holonomy-u}.(d), each $\delta_{r,n}:=e^{i f_{r,n}t}\al_{r,n}$, $(r,n)\in J$ is a flat section of $\cI^\perp$, namely, it satisfies the following equations
	\[
	P_t\delta_{r,n}(t)=\delta_{r,n}(t),\quad
	P_t\dot{\delta}_{r,n}(t)=0.
	\]
	Suppose by contradiction that $P$ has a bounded derivative. Differentiating the first equation and plugging from the second gives
	\[
	\dot{P}_t\delta_{r,n}(t)=\dot{\delta}_{r,n}(t).
	\]
	However, the ratio
	\begin{multline}
	\frac{\left\|\dot{\delta}_{r,n}(t)\right\|}{\left\|\delta_{r,n}(t)\right\|}
	=
	\frac{\left\|if_{r,n}\al_{r,n}(t)+\dot{\al}_{r,n}(t)\right\|}{\left\|\al_{r,n}(t)\right\|}
	=
	\left(\frac{\sum_{q\in\bN}\om_{r+kq,n}^{-1}\e^{2q}(f_{r,n}-q)^2}{\sum_{q\in\bN}\om_{r+kq,n}^{-1}\e^{2q}}\right)^{\frac{1}{2}}\\
	=\left(f_{r,n}^2-2f_{r,n}\frac{\sum_{q\in\bN}\om_{r+kq,n}^{-1}\e^{2q}q}{\sum_{q\in\bN}\om_{r+kq,n}^{-1}\e^{2q}}+\frac{\sum_{q\in\bN}\om_{r+kq,n}^{-1}\e^{2q}q^2}{\sum_{q\in\bN}\om_{r+kq,n}^{-1}\e^{2q}}\right)^{\frac{1}{2}}\label{long}
	\end{multline}
	asymptotically behaves like $n^{\frac{1}{2}}$ as $n\ra\infty$, hence $\dot{P}_t$ would be unbounded. 
	Here are more details.
	By Lemma \ref{lm} and the asymptotic formula for $f_{r,n}$ in (\ref{infty}), the three consecutive terms $a$, $b$ and $c$ in the last expression $\left(a-b+c\right)^{1/2}$ in (\ref{long}) asymptotically behave like $a_2n^2+a_1n$, $b_2n^2+b_1n$ and $c_2n^2+c_1n$, where $a_j$, $b_j$ and $c_j$ are nonzero constants (with respect to $n$) satisfying $a_2-b_2+c_2=0$ and $a_1-b_1+c_1\neq 0$. (Here by saying that $a$ behaves asymptotically like $a_2n^2+a_1n$, we mean that $a\approx n^2$, $a-a_2n^2\approx n$ and $a-a_2n^2-a_1n\ll 1$. Likewise for $b$ and $c$.)
	This shows that $\left(a-b+c\right)^{1/2}$ asymptotically behaves like $n^{1/2}$.
	This contradiction shows that $P$ does not have a bounded derivative.
\end{proof}

We discuss a possible way to fix this problem by using weights to compensate for differentiation \cite{beatrous}.

More precisely, viewing the Drury-Arveson space $H_2^2=\cH^{(-2)}_2$ as a member of the Besov-Sobolev scale $\cH^{(s)}_2$, $s\in\bR$ of Hilbert spaces\footnote{Defined in Appendix \ref{appendix}.}, we have:

\begin{theorem}\label{diffloss}
(a)
The modification $\widetilde{P}:\bR\ra B\left(\cH^{(-2)}_{2},\cH^{(4)}_{2}\right)$ of $P$ is first-order differentiable, where $\widetilde{P}_t$ is the composition of $P_t$ with the inclusion $\cH^{(-2)}_{2}\hookrightarrow\cH^{(4)}_{2}$.

(b)
Given a positive integer $l$ and a positive real $\sigma$, the modification 
\[
\widetilde{P}:\bR\ra B\left(\cH^{(-2)}_{2},\cH^{(2l+1+\sigma)}_{2}\right)
\]
of $P$ is $l$-th order differentiable, where $\widetilde{P}_t$ is the composition of $P_t$ with the inclusion $\cH^{(-2)}_{2}\hookrightarrow\cH^{(2l+1+\sigma)}_{2}$.
\end{theorem}
\begin{proof}
(a)
	We have a corresponding version of Theorem \ref{U} for $\cH^{(4)}_{2}$ instead of $H_2^2=\cH^{(-2)}_{2}$, where $\om_{r+kq,n}$ is replaced by  
	\[
	\widetilde{\om}_{r+kq,n}
	:=\left\| z_1^{r+kq}z_2^n\right\|_{\cH^{(4)}_{2}}
	=\frac{(r+kq)!n!6!}{(r+kq+n+6)!}
	=S(n)\om_{r+kq,n+6},
	\]
	and
	\[
	S(n):=\frac{6!n!}{(n+6)!}\approx n^{-6}.
	\]	
	Let 
	\[
	\left\{e_{m,n}:=\om_{m,n}^{-\frac{1}{2}}z_1^mz_2^n\ : \ (m,n)\in\bN^2\right\}
	\quad\text{and}\quad
	\left\{\widetilde{e}_{m,n}:=\widetilde{\om}_{m,n}^{-\frac{1}{2}}z_1^mz_2^n\ : \ (m,n)\in\bN^2\right\}
	\]
	be the time-independent standard orthonormal bases of $\cH^{(-2)}_{2}$ and $\cH^{(4)}_{2}$, respectively.
	We first compute the matrix coefficients of $\widetilde{P}_t$ with respect to these bases.
	Note that for each $(m,n)\in\bN^2$, $e_{m,n}$ is orthogonal to all members of the orthonormal frame $\widetilde{\beta}$ except for $\widetilde{\be}_{r,n}$, where 
	\[
	m=kQ+r,\quad
	Q,r\in\bN,\quad
	0\leq r<k
	\]
	is the division of $m$ by $k$.
Therefore, $\widetilde{P}_t\left(e_{m,n}\right)=
\left\lang e_{m,n},\be_{r,n}\right\rang\be_{r,n}$ is given by
\begin{multline*}
\frac{\e^{Q}e^{i Q t}\sum_{q\in\bN}\om_{r+kq,n}^{-1}\widetilde{\om}_{r+kq,n}^{\frac{1}{2}}\e^{q}e^{-i q t}\widetilde{e}_{r+k q,n}}{\om_{m,n}^{\frac{1}{2}}\left\|\al_{r,n}\right\|^{2}}=
\frac{\sum\limits_{q}\om_{m,n}^{-\frac{1}{2}}\om_{r+kq,n}^{-1}\widetilde{\om}_{r+kq,n}^{\frac{1}{2}}\e^{Q+q}e^{i\left(Q-q\right)t}\widetilde{e}_{r+kq,n}}{\sum\limits_{q}\om_{r+kq,n}^{-1}\e^{2q}}\\
=\sqrt{S(n)}\frac{\sum\limits_{q}\om_{m,n}^{-\frac{1}{2}}\om_{r+kq,n}^{-1}\om_{r+kq,n+6}^{\frac{1}{2}}\e^{Q+q}e^{i\left(Q-q\right)t}\widetilde{e}_{r+kq,n}}{\sum\limits_{q}\om_{r+kq,n}^{-1}\e^{2q}}.
\end{multline*}
This shows that the formal matrix $\dot{\widetilde{P}}_t$ of entry-by-entry differentiation of $P_t$ is given by
	\[
	\dot{\widetilde{P}}_t\left(e_{m,n}\right)
	:=
	\sqrt{S(n)}\frac{\sum\limits_{q}\om_{m,n}^{-\frac{1}{2}}\om_{r+kq,n}^{-1}\om_{r+kq,n+6}^{\frac{1}{2}}\e^{Q+q}i(Q-q)e^{i\left(Q-q\right)t}\widetilde{e}_{r+kq,n}}{\sum\limits_{q}\om_{r+kq,n}^{-1}\e^{2q}}.
	\]
From this expression, the Hilbert-Schmidt norm of $\dot{\widetilde{P}}_t$,
\[
\left\|\dot{\widetilde{P}}_t\right\|_{\mathrm{HS}}^{2}
:=\sum\limits_{n,r}S(n)\frac{\sum\limits_{Q,q}\om_{r+kQ,n}^{-1}\om_{r+kq,n}^{-2}\om_{r+kq,n+6}\e^{2Q+2q}(Q-q)^2}{\left(\sum\limits_{q}\om_{r+kq,n}^{-1}\e^{2q}\right)^2},
\]
is bounded above by
\[
\sum\limits_{n,r}\frac{S(n)\sum\limits_{Q,q}\om_{r+kQ,n}^{-1}\om_{r+kq,n}^{-1}\e^{2Q+2q}\left(Q^2+q^2\right)}{\left(\sum\limits_{q}\om_{r+kq,n}^{-1}\e^{2q}\right)^2}
=
	\sum\limits_{n,r} \frac{2 S(n)\sum\limits_{q}\om_{r+kq,n}^{-1}\e^{2q}q^2}{\sum\limits_{q}\om_{r+kq,n}^{-1}\e^{2q}}\approx
	\sum\limits_{n}S(n)n^{2},
\]
a finite number.
This shows that $\dot{\widetilde{P}}_t$ is bounded (in operator norm) \cite[2.8.4]{arveson-spectral}.
	With the same line of arguments along with Taylor's theorem, for any $h\in\bR$, we have
	\begin{equation*}
	\left\|\widetilde{P}_{t+h}-\widetilde{P}_t-h\dot{\widetilde{P}}_t\right\|_{\mathrm{HS}}^{2}
	\leq
	\sum\limits_{n,r}S(n)\frac{\sum\limits_{Q,q}\om_{r+kQ,n}^{-1}\om_{r+kq,n}^{-2}\om_{r+kq,n+6}\e^{2Q+2q}(Q-q)^4\left(\frac{h^2}{2!}\right)^2}{\left(\sum\limits_{q}\om_{r+kq,n}^{-1}\e^{2q}\right)^2}
	\end{equation*}
	\begin{equation*}
	\ll
	\sum\limits_{n,r} S(n)h^4 \frac{\sum\limits_{q}\om_{r+kq,n}^{-1}\e^{2q}q^4}{\sum\limits_{q}\om_{r+kq,n}^{-1}\e^{2q}}
	\ll
	\sum_{n}S(n)h^4n^{4}
	=
	h^4\sum\limits_{n}n^{-2}.
	\end{equation*}
	This finishes the proof that $\widetilde{P}$ is first-order differentiable. 
	
(b) 
	Imitating the proof in Part (a), we set
	\[
	\widetilde{\om}_{r+kq,n}
	:=\left\| z_1^{r+kq}z_2^n\right\|_{\cH^{(2l+1+\sigma)}_{2}}
	=\frac{(r+kq)!n!(2l+3+\sigma)!}{(r+kq+n+2l+3+\sigma)!}
	=S(n)\om_{r+kq,n+2l+3+\sigma},
	\]
	where
	\[
	S(n):=\frac{(2l+3+\sigma)!n!}{(n+2l+3+\sigma)!}\approx n^{-2l-3-\sigma}.
	\]	
Given $j=1,\ldots,l$ and $h\in\bR$, we have estimates
	\[
	\left\|\frac{d^j\widetilde{P}_t}{dt^j}\right\|_{\mathrm{HS}}^{2}
	\ll
	\sum\limits_{n}S(n)n^{2j}
	\approx
	\sum\limits_{n}n^{-3-\sigma-2(l-j)}
	<\infty,
	\]
	\[
	\left\|\frac{d^{j-1}\widetilde{P}_{t+h}}{dt^{j-1}}-\frac{d^{j-1}\widetilde{P}_{t}}{dt^{j-1}}-h\frac{d^{j}\widetilde{P}_{t}}{dt^{j}}\right\|_{\mathrm{HS}}^{2}
	\ll
	\sum\limits_{n}S(n)h^{2j+2}n^{2j+2}
	\approx
	h^{2j+2}\sum\limits_{n}n^{-1-\sigma-2(l-j)},
	\]
	which implies that $\widetilde{P}$ is $l$-th order differentiable.
\end{proof}

\begin{remark}
Recall the identification $\cH^{(s)}_{m}= W^{-\frac{s}{2}}_{\mathrm{hol}}(\bB^m)$, $s\in\bR$, between Besov-Sobolev and Bergman-Sobolev spaces mentioned in Appendix \ref{appendix}. 
Theorem \ref{diffloss}.(b) says that by taking $l$-th order derivative of $P$, we lose differentiability by order no worse than $l+2$.
	We do not know whether this estimate of differentiability loss is optimal.
\end{remark}

\begin{remark}\label{rmk:der-P}
	Suppose a section $\xi\in C^\infty(\bR,\cI^\perp)$.
	Proposition \ref{diffloss}.(b) shows that $\frac{d^lP_t}{dt^l}\left(\xi(t)\right)$ lives in $\cH^{(2l+1+\sigma)}_2$.
	Similar arguments show that for each $s\leq -2$, if $\xi(t)\in\cH^{(s)}_2$, then $\frac{d^lP_t}{dt^l}\left(\xi(t)\right)$ in fact lives in $\cH^{(s+2l+3+\sigma)}_2$.
\end{remark}

We are now ready to conclude with the smoothness of the connection $D$. 
\begin{corollary}
Let $\mathcal{S}$ denote the set of all sections $\xi\in C^\infty(\mathbb{R}; \mathcal{I}^\perp)$ such that for each $t$, $\xi(t)$ and all its time derivatives live in $\bigcap_{s\in\bR}\cH^{(s)}_2$. The connection $D$ defined by (\ref{connection-toy}) maps $\mathcal{S}$ to itself.
\end{corollary}

\begin{proof}
Suppose that $\xi$ is from $\mathcal{S}$.  Then $D\xi= P_t (\frac{d}{dt} \xi)$ by Equation (\ref{connection-toy}). By the definition of $\mathcal{S}$, $\frac{d}{dt}\xi$ belongs to $\mathcal{S}$. It follows from Theorem \ref{diffloss} that $P_t(\frac{d}{dt}\xi)$ is in $\bigcap_{s\in\bR}\cH^{(s)}_2$. 

We compute the derivative of $P_t(\frac{d}{dt}\xi)$, i.e.
\[
\frac{d}{dt}\left( P_t \Big(\frac{d}{dt} \xi\Big)\right)=\frac{dP_t}{dt}\Big( \frac{d}{dt} \xi\Big)+P_t \Big(\frac{d^2}{dt^2}\xi\Big).
\]
It follows from Remark \ref{rmk:der-P} that  both $\frac{dP_t}{dt}\big( \frac{d}{dt} \xi\big)$ and $P_t (\frac{d^2}{dt^2}\xi)$ belong to $\bigcap_{s\in\bR}\cH^{(s)}_2$. Therefore, $\frac{d}{dt}\big(P_t(\frac{d}{dt}\xi)\big)$ lives in $\bigcap_{s\in\bR}\cH^{(s)}_2$.  

By induction, one easily proves that $\frac{d^l}{dt^l} \big(P_t(\frac{d}{dt}\xi)\big)$ belongs to $\bigcap_{s\in\bR}\cH^{(s)}_2$. And $D(\xi)$ belongs to $\mathcal{S}$. 
\end{proof}

\section{Outlook} \label{sec:outlook}
It is natural to ask whether Theorems \ref{frame}, \ref{thm:holonomy-u}, and \ref{conjugation} have a natural extension to more general principal ideals. In the following, we propose a series of problems attempting to generalize the results in Section \ref{sec:power}. 

Suppose a polynomial $f\in A=\bC[z_1,\ldots,z_m]$ which vanishes at the origin, and it has the origin as an isolated critical point.
In geometric terms, the origin is an isolated singularity of the hypersurface $V(f)\sub\bC^m$.
Consider the family of principal ideals
\[
I(t):=\lang f-\e e^{it}\rang\sub A,\quad
t\in\bR,
\]
where $\e$ is a fixed, sufficiently small, positive real number. 
We think of $t$ as the time variable.
Let $P_t$ be the orthogonal projection in $H_m^2$ onto $I(t)^\perp$.
Let
\[
p:\mathcal{I}^{\perp}\ra\bR,\quad
\cI^\perp:=\biguplus\left\{I(t)^{\perp}\sub  H_m^2:t\in\bR\right\}\sub\bR\times H_m^2,\quad
p\left(I(t)^\perp\right)=\{t\},
\]
and 
\[
P:\bR\ra B\left( H^2_m \right),\quad
P:=(P_t)
\]
be respectively the assembly of Hilbert spaces $I(t)^\perp$ and orthogonal projections $P_t$ into a rough Hilbert bundle and a rough map between Banach spaces.
Topologize $\cI^\perp\sub\bR\times H_m^2$ with the subspace topology.

We would like to ask the following question about $p$.
\begin{problem}\label{conjecture-smoothness-p}
Is $p$ a smooth Hilbert bundle\footnote{$C^2$-smoothness is enough for our purposes \cite{klingenberg,km-DG,lang}.}?
\end{problem}
At the moment, we can only prove that $p$ is a smooth Hilbert bundle for the special case in Section \ref{sec:power} (see Theorem \ref{frame}). Note that since the base space of $p$ is contractible, even the weaker property that $p$ is a topological vector bundle implies that it is trivial \cite[IV.2.5]{hirsch}, \cite[Corollary 1]{kuiper}, hence automatically smooth, and this smooth structure is unique up to smooth vector bundle isomorphisms \cite[IV.3.5]{hirsch}.
Let us proceed with assuming that Problem \ref{conjecture-smoothness-p} has a positive answer. Denote the set of all (smooth) sections of $p$ by $C^{\infty}(\bR;\cI^{\perp})$.

Unfortunately,  the family of projections $P$ is not smooth in general (see Section \ref{nonsmoothness}, Proposition \ref{prop:nonsmooth}).
Thinking of $P$ as a rough connection between nearby fibers $I(t)^\perp$, the imitation of the standard construction of the Hermitian connection for subbundles of Hilbert bundles \cite[Example 1.5.14]{klingenberg},\cite[Volume II, Page 540]{taylor-PDE} gives us a rough covariant derivative:
\begin{equation}
	D\xi(t)=P_t\left(\frac{d\xi}{dt}\right),\quad\xi\in C^{\infty}(\bR;\cI^\perp).\label{connection}
\end{equation}

Note that $D$ is called a covariant derivative because it satisfies the Leibniz rule
\[
D(g\xi)(t)=g'(t)\xi(t)+g(t)D(\xi)(t),\quad
\forall g\in C^{\infty}(\bR;\bC),\quad
\forall \xi\in  C^{\infty}(\bR;\cI^{\perp}).
\]
The $D$-flat sections of $p$ are those $\xi\in C^{\infty}(\bR;\cI^{\perp})$ which satisfy the evolution equation
\begin{equation}
	D\xi(t)=0,\quad \forall t\in\bR.\label{ODE}
\end{equation}

This leads us to the problem about solving Equation (\ref{ODE}). 
\begin{problem}\label{conjecture-ODE}
Does the parallel transport equation (\ref{ODE}) has a unique solution on $t\in\bR$ for each initial value $\xi(0)\in I(0)^\perp$?
\end{problem}

Suppose that the operator $U_t: I(0)^\perp \to I(t)^\perp$ is the solution operator to Equation (\ref{ODE}). We are especially interested in
\[
U:=U_{2\pi}\in B(I(0)^{\perp}),
\]
which can be viewed as a noncommutative analogue of the Milnor monodromy map $h:F_0\ra F_0$ in \cite[Page 67]{milnor}.
We would like to know the answer to the following problem. 

\begin{problem}\label{conjecture-conjugation}
Does	$U$ act by conjugation on the Toeplitz algebra $\fT_{I(0)}$ in the sense that $U\fT_{I(0)}U^{\ast}\sub\fT_{I(0)}$?
\end{problem}

A positive answer to Problem \ref{conjecture-conjugation}  will induce a map $K\left(\fT_{I(0)}\right)\ra K\left(\fT_{I(0)}\right)$ at the $K$-homology level. Such a map should be viewed as the analytic analogue of the monodromy operator introduced and studied by differential topologists, e.g. \cite{milnor, looijenga}.

Besides the examples $z_1^k$ in Section \ref{sec:power}, we are also able to solve Problems \ref{conjecture-smoothness-p}, \ref{conjecture-ODE} and \ref{conjecture-conjugation} for monomials $z_1^kz_2^l$ with similar methods.  Unfortunately, our current methods in Section \ref{sec:power} are not sufficiently developed to solve Problems \ref{conjecture-smoothness-p}, \ref{conjecture-ODE} and \ref{conjecture-conjugation} for more general ideals. 

We end this paper by pointing out that it is interesting to study the asymptotic behavior of the unitary operator $U$ when $\e\ra 0$. More specifically, note that in Theorem \ref{conjugation} appears the \textit{phase factor} $\exp\frac{2\pi i F}{k(1-F)}$, $F=\e^{\frac{2}{k}}$. When $\e\ra 0$, this factor varies like
	\[
	\exp\left(\frac{2\pi i}{k}\e^{\frac{2}{k}}\right).
	\]
	For the example $f:=z_1z_2\in\bC[z_1,z_2]$, our computations (with similar methods as in Section \ref{sec:power}) show that the phase factor equals
	\[
	\exp\left(2\pi i\frac{1-\sqrt{1-4\e^2}}{\sqrt{1-4\e^2}}\right)
	=\exp\left(4\pi i \e^2+O\left(\e^4\right)\right).
	\]
	It is desirable to understand these phase factors in the general case and study  their connections to the corresponding polynomials.

\appendix
\section{Some analytic Hilbert spaces}\label{appendix}
In this appendix, we review the definitions of some famous analytic Hilbert spaces used in this paper.

\begin{enumerate}[leftmargin=*]
	\item
	Let $\Omega\sub\bC^m$ be a bounded domain with smooth boundary.
	\begin{itemize}[leftmargin=*]
		\item
		$W^{s}_{\mathrm{hol}}(\Om)$, $s\in\bR$ is the Bergman-Sobolev space consisting of all holomorphic functions in the $L^2$ Sobolev space $W^s(\Omega)$. (See \cite{ee,ligocka-86,beatrous}.)
		These are also known as the holomorphic Sobolev spaces.
		\label{Bergman-Sobolev}

		\item
		$W^{s}_{\mathrm{hol}}(\pa\Omega)$, $s\in\bR$ is the Hardy-Sobolev space consisting of all functions in the $L^2$ Sobolev space $W^s(\pa\Omega)$ whose Poisson extension to $\Om$ is not only harmonic but also holomorphic.
		Alternatively, it is the closure in $W^s(\pa\Omega)$ of the boundary values of holomorphic functions on $\Om$ which are continuous up to the boundary \cite{b,ee}.
		Note that $W^{s}_{\mathrm{hol}}(\pa\Omega)$ is isometrically isomorphic to the Bergman-Sobolev space $W^{s+\frac{1}{2}}_{\mathrm{hol}}(\Omega)$ through the Poisson extension and the trace map \cite{ee}.\label{Hardy-Sobolev}
		
		\item
		$H^2(\pa\Omega):=W^{0}_{\mathrm{hol}}(\pa\Omega)$ is the Hardy space \cite{krantz,stein,upmeier,coburn}.

		\item 
		$L_{a,s}^2(\Omega)$, $s>-1$ is the weighted Bergman space consisting of holomorphic functions $f$ on $\Omega$ such that $\int_{\Omega}|f(z)|^2 \rho(z)^{s}dV(z)<\infty$, where $\rho(z)$ is a positively signed smooth defining function for $\Omega$ (equivalently, the distance function $\mathrm{dist}(z,\pa\Om)$), and $dV(z)$ is the Lebesgue measure on $\Om$ normalized such that $\int_{\Omega}\rho(z)^{s}dV(z)=1$. (See \cite{beatrous,ee,zhu-FT}.)
		Note that $L_{a,s}^2(\Omega)=W^{-\frac{s}{2}}_{\mathrm{hol}}(\Om)$ as sets with equivalent norms \cite{ligocka-87,ee,beatrous}.

		\item
		$L_{a}^2(\Omega):=L_{a,0}^2(\Omega)$ \label{ahs} is the (unweighted) Bergman space \cite{krantz,upmeier,taylor-PDE}.
	\end{itemize}
	
	\item
	$H_m^2$ is the Drury-Arveson space of analytic functions on $\bB^m$, the one with the reproducing kernel $\left(1-\lang z,w\rang\right)^{-1}$. 
	(See \cite{arveson-dilation}, \cite[Chapter 41]{alpay}.)
	It has the standard orthonormal basis $(n!/|n|!)^{-1/2}z^n$, $n\in\bN^m$. 
	It is also known as the $m$-shift or symmetric Fock space.
	\label{Drury-Arveson}

	\item
	$\cH^{(s)}_m$, $s\in\bR$ is the Besov-Sobolev space of analytic functions on $\bB^m$, the one with the reproducing kernel
	\[
	K_s(z,w)
	:=\begin{cases}
	(1-\lang z,w\rang)^{-s-m-1},&\quad s>-m-1,\\
	(-s-m)^{-1}F\left(1,1;1-s-m;\lang z,w\rang\right),&\quad s\leq -m-1,
	\end{cases}
	\]
	where $F(a,b;c;\z):=\sum_{q\in\bN}\frac{(a)_q(b)_q}{(c)_q q!}\z^q$ is the hypergeometric function, and $(x)_y:=\Gamma(x+y)/\Gamma(x)$ is the Pochhammer symbol. (See \cite{beatrous2,zhu-zhao,ee,arcozzi-etal,zhu-FT}; our parameter $s+m+1$ is $q$ in \cite{beatrous2}, $\al+m+1$ in \cite{zhu-zhao,ee}, and $2\sigma$ in \cite{arcozzi-etal}; \cite{zhu-FT} only studies the $s=-m-1$ case.)
	$\cH^{(s)}_m$ has the standard orthonormal basis $\om_s(n)^{-1/2}z^n$, $n\in\bN^m$, where
	\[
	\omega_s(n)
	:=\begin{cases}
	\frac{n!(s+m)!}{(|n|+s+m)!},&\quad s>-m-1,\\
	\frac{n!(-s-m)_{|n|+1}}{\left(|n|!\right)^2},&\quad s\leq -m-1.
	\end{cases}
	\]
	Note that 
	\[
	\om_s(n)\approx\frac{n!}{|n|! \left(|n|+1\right)^{s+m}}
	\]
	for each $s\in\bR$.
	(We will not need the reproducing kernel, but this equivalent norm is enough for our purposes.)
	We have the identifications:
	{\small
		\[
		\cH_{m}^{(s)}
		=\begin{cases}
		\text{the Bergman-Sobolev space} \ W^{-\frac{s}{2}}_{\mathrm{hol}}(\bB^m) \ \text{(as sets with equivalent norms)},& s\in\bR,\\
		\text{the Hardy-Sobolev space} \ W^{-\frac{s+1}{2}}_{\mathrm{hol}}(\pa\bB^m) \ \text{(as sets with equivalent norms)},& s\in\bR,\\
		\text{the Drury-Arveson space} \
		H_m^2 \ \text{(as sets with equal norms)},& s=-m,\\
		\text{the Hardy space} \
		H^2(\pa\bB^m) \ \text{(as sets with equal norms)},& s=-1,\\
		\text{the weighted Bergman space} \
		L^2_{a,s}(\bB^m) \ \text{(as sets with equal norms)},& s>-1.
		\end{cases}
		\]
	}
	\label{Besov-Sobolev}

\end{enumerate}

\end{document}